\documentclass[11pt,a4paper]{article}
\usepackage[utf8]{inputenc}
\usepackage[T1]{fontenc}
\usepackage{amsfonts,amsmath,amssymb,amsthm,amscd}
\usepackage{mathtools}
\usepackage{natbib}
\usepackage{authblk}
\usepackage{bm, bbm}

\usepackage{fullpage}

\newtheorem{theorem}{Theorem}[section]
\newtheorem{proposition}{Proposition}[section]
\newtheorem{lemma}{Lemma}[section]
\newtheorem{corollary}{Corollary}[section]

\newtheorem{assumption}{Assumption}[section]

\theoremstyle{remark}
\newtheorem*{remark}{Remark}

\DeclareMathOperator{\E}{\mathbb{E}}
\let\P\relax
\DeclareMathOperator{\P}{\mathbb{P}}
\DeclareMathOperator{\Var}{Var}
\DeclareMathOperator{\Cov}{Cov}

\let\Re\relax
\DeclareMathOperator{\Re}{\mathrm{Re}}

\DeclareMathOperator{\sinc}{sinc}

\newcommand{\1}{\mathbbm{1}}

\newcommand{\N}{\mathbf{N}}

\newcommand{\M}{\mathbf{M}}
\newcommand{\Z}{\mathbf{Z}}

\DeclarePairedDelimiter{\parentheses}{(}{)}
\newcommand{\bigO}{\ensuremath{\mathop{}\mathopen{}\mathcal{O}\mathopen{}\parentheses}}
\newcommand{\smallO}{\ensuremath{\mathop{}\mathopen{}o\mathopen{}\parentheses}}

\newcommand{\dd}{\mathrm{d}}

\newcommand{\Nn}{\mathbb{N}}
 
\newcommand{\Rr}{\mathbb{R}} \newcommand{\R}{\mathbb{R}}

\author[1]{Ousmane Boly}
\author[2,3]{Felix Cheysson}
\author[1]{Thi Hien Nguyen}

\affil[1]{Laboratoire AGM, UMR CNRS 8088, CY Cergy Paris Université, 95000 Cergy, France. Email:\{ousmane.boly, thi-hien.nguyen5\}@cyu.fr}
\affil[2]{LPSM, CNRS UMR 8001, Sorbonne Université, Université Paris Cité, Paris, France.}
\affil[3]{LAMA, CNRS UMR 8050, Université Gustave Eiffel, 77420 Champs-sur-Marne, France. Email: felix.cheysson@univ-eiffel.fr}

\title{Mixing properties for multivariate Hawkes processes}   
\begin{document}
\pagestyle{plain}
\maketitle

\begin{abstract}
	Properties of strong mixing have been established for the stationary linear Hawkes process in the univariate case, and can serve as a basis for statistical applications.
	In this paper, we provide the technical arguments needed to extend the proof to the multivariate case.
	We illustrate these properties by establishing a functional central limit theorem for multivariate Hawkes processes.	
\end{abstract}
\section{Introduction}

Linear Hawkes processes form a versatile family of models for point processes that exhibit self-excitation properties: in addition to a constant rate of arrivals, each new arrival increases the rate at which future arrivals occur.
This additive dependence structure can also be decomposed such that the process can be viewed as a branching process with immigration, thus admitting a Poisson cluster structure where each cluster is itself a branching process without immigration \citep{Hawkes1974}.
While they were initially introduced by \cite{Hawkes1971} in order to model seismic aftershocks \citep{Ogata1988}, Hawkes processes have become increasingly popular due to their wide applicability to various other fields including but not limited to finance \citep{Embrechts2011}, ecology \citep{gupta2017}, neurophysiology \citep{Chornoboy1988}, and social network analysis \citep{Zannettou2018}.
The study of Hawkes processes in itself is also of interest, as theoretical advances in the past decades have paved the way to the development for statistical inference methods \citep{Hansen2015, Bacry2020} and thus to the popularity of the model.


%
%

In particular, works interested in the ergodic and mixing properties of the Hawkes process play an important role, for they can provide strong moment inequalities and coupling methods useful to prove asymptotic properties.
It is well known that the linear Hawkes process is ergodic due to its Poisson cluster structure \citep{Westcott1971}.
When the reproduction function has compact support, \cite{Reynaud-Bouret2006} used this structure to derive exponential inequalities which specify the rate of convergence in the ergodic theorem.
Recently, \cite{Graham2021} extended these results by proving regenerative properties for the linear Hawkes process, even when the reproduction function has unbounded support.
Issues arise for the nonlinear Hawkes process --- in which the rate at which future arrivals occur is non-linearly dependent on the past arrivals --- since the Poisson cluster structure of the process no longer exists.
While the existence and stationarity of the process can be proven under general sets of assumptions \citep{Bremaud1996}, studies of the mixing properties of the process have only been established under specific conditions.
In \cite{Costa2020}, by assuming a bounded support for the reproduction function (which may take negative values), the authors obtained exponential concentration inequalities using renewal techniques, further extending the results from \cite{Reynaud-Bouret2006}.
Without the bounded support assumption, but supposing the reproduction function is exponential, \cite{Dion2021} showed that the non-linear Hawkes process is exponentially $\beta$-mixing, relying on the general theory of Markov processes.


In this work, we focus on establishing rates of convergence for the $\alpha$-mixing of the linear Hawkes process, using the Poisson cluster decomposition of the process.
Initially used to prove asymptotic properties for time series and random fields (see e.g. \citealp{Rosenblatt1956}), the $\alpha$-mixing property can be straightforwardly extended to the field of point processes \citep{Westcott1972, Poinas2019}.
In the univariate context, \cite{Cheysson2022} already derived a polynomial rate of $\alpha$-mixing under tail assumptions for the reproduction function of the Hawkes process.
However, their proof cannot be directly extended to the multivariate case, as a critical argument of independence within each Poisson cluster, that is between the arrivals' times and the number of arrivals in any given generation of the branching representation of the cluster, does not hold when multiple components can mutually excit each other.
This is solved by establishing an exponential inequality in the multitype branching process, which allows us to bound the covariance measure of the branching process.

In turn, the $\alpha$-mixing of the process can serve as a basis for statistical applications, for example by using central limit theorems (CLTs) that already exist in the literature \citep{Doukhan1994}.
To provide useful tools for statistical applications related to the multivariate Hawkes process, we enunciate both a classical and a functional CLTs in this context, and show that the parameters appearing in these CLTs can be easily calculated in practice, using the spectral representation of the process.
Indeed, while the covariance structure of the Hawkes process is often non-explicit, except for exponential reproduction functions \citep{DaFonseca2014, DaFonseca2015}, its Bartlett spectrum (i.e. the Fourier transform of the autocovariance of the process) has an explicit and simple form \citep{Daley2003}.
Specifically, this yields a simple estimator of the long run variance of the process which appears naturally in the CLTs.


Section \ref{sec:notation} introduces the notation we use throughout and the Hawkes model.
In Section \ref{sec:main}, we establish the main results of our contribution, namely an $\alpha$-mixing property and a functional CLT for the multivariate Hawkes process.
Proofs are postponed to subsequent sections, respectively in Section \ref{sec:proof} and \ref{sec:proof_fclt}.


\section{Main definitions}\label{sec:notation}

\subsection{Notation}

Throughout this paper, we consider a probability space $(\Omega, \mathcal{F},\mathbb{P})$ and the measurable space $(\mathfrak N, \mathcal N)$ of locally finite counting measures on $\mathbb{R}$ with values in $\Nn^d$, with $d \ge 1$.
We denote by $\mathcal{B}(A)$ the Borel $\sigma$-algebra generated over a subset $A \subset \mathbb R$, and for a multivariate point process $\N = (N_1, \ldots, N_d)$ \textit{i.e.} a measurable map from $(\Omega, \mathcal{F},\mathbb{P})$ to $(\mathfrak N, \mathcal N)$, by $\mathcal{E}(A)$ the cylindrical $\sigma$-algebra generated by $\N$ on A and defined by
\[\mathcal{E}(A):= \sigma (\{ \N \in \mathfrak N: \N(B)=\mathbf m\}, B\in\mathcal{B}(A), \mathbf m \in \Nn^d). \]
Finally, denote by $\mathbf N_{|A}$ the point process restricted to $A$, that is the process defined by $\mathbf N_{|A}(B) = \mathbf N(A \cap B)$ for all $B \in \mathcal B(\mathbb R)$.

We also introduce some notation for function and vector norms. For a function $f: \mathbb R \to \mathbb R$, denote its $L^1$-norm by $\lVert f \rVert_{L^1}$,
\begin{equation*}
	\lVert f \rVert_{L^1} = \int_{\mathbb R} \lvert f(x) \rvert \dd x.
\end{equation*}
For a vector $\mathbf u = (u_1, \ldots, u_n) \in \mathbb R^n$, denote its $\ell^1$-norm by $\lVert \mathbf u \rVert_1$,
\begin{equation*}
	\lVert \mathbf u \rVert_1 = \sum_{k=1}^n \lvert u_k \rvert,
\end{equation*}
and its supremum norm by $\lVert \mathbf u \rVert_\infty$,
\begin{equation*}
	\lVert \mathbf u \rVert_\infty = \sup_{1 \le k \le n} \lvert u_k \rvert.
\end{equation*}

\subsection{The multivariate linear Hawkes process}

The multivariate Hawkes process $\N = (N_1, \ldots, N_d)$ can be defined as a point process such that
\begin{itemize}
	\item For all $1 \le j \le d$, the conditional intensity $\lambda_j(t)$ of the component $N_j$ --- satisfying 
	\[\P\left(N_j\big((t, t+h)\big) = 1 \mid \mathcal{F}_{t^-}\right) = \lambda_j(t) h + \smallO{h},\]
	with $\mathcal{F}_t$ the natural filtration of the process $\N$ --- is a left-continuous stochastic process given by 
	\begin{align*}
	\lambda_j(t)&=\eta_j + \sum_{i=1}^{d}\int_{-\infty}^{t}h_{ij}(t-u)N_i(\mathrm du) \\
	&= \eta_j + \sum_{i=1}^d \sum_{\{n: T_i^n <t\}} h_{ij}(t-T_i^n),
	\end{align*}
	where $\eta_j > 0$ is the baseline intensity of $N_j$, $h_{ij}: \R^+\to\R^+$ the reproduction function from component $i$ to $j$, and $\{ T_i^n \}_n$ are the atoms of $N_i$.
	\item The point process is orderly: for all $1 \le i\neq j \le d$, $N_i$ and $N_j$ never jump simultaneously,
	\[\P\left(\lVert \mathbf N\big((t, t+h)\big) \rVert_1  \geq 2 \mid \mathcal{F}_{t^-}\right) = \smallO{h}. \]
\end{itemize}

The multivariate Hawkes process $\N$ can also be seen as a Poisson cluster process (see \citealp{Hawkes1974}), where clusters are mutually independent branching processes generated recursively as follows:
\begin{itemize}
	\item Cluster centres, also called immigrants of the process, form the generation 0 of the branching processes. Immigrants of type $j$ arrive according to a homogeneous Poisson process with rate $\eta_j$. 
	\item Then, an individual $T_{ki}^n$ of generation $k$ and type $i$ generates offsprings of generation $k+1$ and type $j$ according to an inhomogeneous Poisson process with intensity $h_{ij}(\cdot - T_{ki}^n)$.
\end{itemize}
For a single branching process, denoting by $Z_{ki}$ the number of individuals of generation $k$ and type $i$, we have that $(Z_{ki})_{k \in \mathbb N, 1 \le i \le d}$ is a multivariate Galton--Watson process. We will make the following assumption, which ensures the existence and stationarity of the multivariate Hawkes process \citep{Bremaud1996}.
\begin{assumption}
	\label{ass:spectral_radius}
	The spectral radius of the reproduction matrix, defined by $\mathbf M \coloneqq (\lVert h_{ij} \rVert_{L^1})_{1 \le i,j \le d}$, is strictly less than $1$.
\end{assumption}

\section{Main results}\label{sec:main}
\subsection{Strong mixing of the multivariate Hawkes process}
For a probability space $(\Omega, \mathcal{F}, \P)$ and $\mathcal{A}, \mathcal{B}$ two sub $\sigma$- algebras of $\mathcal{F}$, the strong mixing coefficient is defined as the measure of dependence between $\mathcal A$ and $\mathcal B$ \citep{Rosenblatt1956}:
\[\alpha(\mathcal{A},\mathcal{B}):= \sup \{\left| \P(A\cap B)-\P(A)\P(B)\right|, A\in\mathcal{A}, B\in\mathcal{B}. \} \]
The definition of the strong mixing coefficient can be adapted to a multivariate point process $\N$, by defining (see \citealp{Poinas2019})
\begin{equation*}
\alpha_{\N}(\tau) = \alpha(\mathcal{E}_{-\infty}^t, \mathcal{E}_{t+\tau}^{+\infty}) 
= \underset{t \in \mathbb{R}}{\mathrm{sup}} 
~\underset{\begin{subarray}{c}
	\mathcal{A} \in \mathcal{E}_{-\infty}^t \\
	\mathcal{B} \in \mathcal{E}_{t+\tau}^\infty
	\end{subarray}}
{\mathrm{sup}}  ~\big| \mathrm{Cov}\big(\mathbbm{1}_\mathcal{A}(\N), \mathbbm{1}_\mathcal{B}(\N)\big) \big|
\end{equation*}
where $\mathcal{E}_a^b$ stands for $\mathcal{E}((a,b])$, i.e the $\sigma$-algebra generated by the cylinder sets on the interval $(a,b]$, and $\1 _\mathcal{A}(\N)$ is the indicator function of the cylinder set $\mathcal{A}$, i.e for any elementary cylinder set $\mathcal{A}_{B,\mathbf n} = \{\N \in\mathfrak N : \N(B) =\mathbf n\}$, $\1 _\mathcal{A}(\N) = 1$ if $\N(B) = \mathbf n$ and $0$ otherwise.
As in the univariate case, the process is said to be strongly mixing if $\alpha_\N(\tau) \to 0$ as $\tau \to \infty$.

In the context of multivariate Hawkes processes, the main result of this contribution is the following theorem.
Theorem \ref{thm:strong_mixing} can serve as a basis for various statistical applications, as examplified by the next subsection.
\begin{theorem}\label{thm:strong_mixing}
	Let $\N$ be a multivariate Hawkes process with reproduction functions $(h_{ij})_{1 \le i,j \le d}$ such that Assumption \ref{ass:spectral_radius} holds true. 
	Further assume that there exists $\beta > 0$ such that the reproduction kernels have finite moment of order $1 + \beta$:
	\begin{equation*}
	\sup_{1 \le i,j \le d} \int_{\mathbb R} t^{1 + \beta} h_{ij}(t) \mathrm dt < \infty.
	\end{equation*}
	Then, for any $0 < \gamma < \beta$, as $\tau \to \infty$,
	\begin{equation}
	\alpha_{\N} (\tau) = \mathcal O\left(\tau^{-\gamma}\right).
	\end{equation}
\end{theorem}

While the proof (postponed to Section \ref{sec:proof}) mostly extends that of \cite{Cheysson2022} to the multivariate case, it also circumvents a lemma in the proof of the aforementioned article: in the univariate Galton-Watson process, the number of individuals in any given generation is independent from the arrival times of all individuals, and vice versa, whereas this is only valid conditionally on the genealogy of individuals in the multivariate case.
This is the heart of the proof and the focus of Lemmas \ref{lem:laplace_Z} and \ref{lem:summability_Z}, which aim to prove that we can make do with exponential inequalities instead of independence when bounding the strong mixing coefficient, thanks to the asymptotic contraction property of the reproduction matrix $\M$, ensured by Assumption \ref{ass:spectral_radius}.

\subsection{An application to statistics}\label{sec:appli}
\subsubsection{A functional central limit theorem}

Suppose that we observe a point process $\N = (N_1, \ldots, N_d)$ from a stationary multivariate Hawkes model with intensities $m_i \ge 0$, and that we are interested in the asymptotic behaviour of statistics of the form 
\begin{equation*}
\sum_{i = 1}^d N_{i}(f_i\big\vert_{\![0,T]}) = \sum_{i=1}^d \int_{[0,T]} f_i(x) N_i(\dd x),
\end{equation*}
or rather their centered version, 
\begin{equation}\label{eqn:partial_sums}
S_T = \sum_{i=1}^d \int_{[0,T]} f_i(u) \bigl(N_i(\dd u) - m_i\dd u \bigr),
\end{equation}
where $f_i: \mathbb R \to \mathbb R$ denote bounded measurable functions ($1 \le i \le d$) and $f_i\big\vert_{\![0,T]}: x \mapsto f_i(x) \1_{[0,T]}(x)$ their restriction to $[0,T]$.
The variance of $S_T$ can be related to the covariance measure of $\N$, whose elements are denoted $C_{ij}$, through the use of Campbell's formula: 
\begin{equation*}
\sigma_T^2 = \Var S_T = \sum_{i = 1}^d \sum_{j=1}^d \iint_{[0,T]^2} f_i(u)f_j(v) C_{ij}(\dd u \times \dd v).
\end{equation*}
Finally, for $0 \le u \le 1$, define
\begin{equation*}
v_T(u) = \inf \{ t \in (0, T): \sigma_t^2 / \sigma_T^2 \ge u \} \quad \text{and} \quad W_T(u) = \sigma_T^{-1} S_{v_T(u)}.
\end{equation*}

The following proposition is a consequence of the functional central limit theorem for nonstationary strongly mixing processes in \citet[Corollary 2.2]{Merlevede2020}, in conjunction with the mixing property derived in Theorem \ref{thm:strong_mixing}.
Its proof is deferred until Section \ref{sec:proof_fclt}.

\begin{proposition}\label{prop:fclt}
	Let $\N$ be a multivariate Hawkes process satisfying the conditions of Theorem \ref{thm:strong_mixing}.
	Assume also that $\sigma_T^2 = T \sigma^2 + o(T)$ as $T \to \infty$, and that each $f_i$ is locally $(2+\delta)$-integrable, with $\beta, \delta > 0$, such that $(\beta-1)\delta > 2$.
	Then, $\{W_T(u), u \in (0,1)\}$ converges in distribution in the Skorokhod space $D([0,1])$ (equipped with the uniform topology) to the standard Brownian motion.
\end{proposition}


Note also that this result implies that $S_T$ satisfies the classical central limit theorem (see \citealp[Section 2.1.3]{Merlevede2020}).
\begin{corollary}\label{cor:clt}
	Let $\N$ be a multivariate Hawkes process as in Proposition \ref{prop:fclt}.
	Then $\sigma_T^{-1} S_T$ converges in distribution to the standard normal distribution.
\end{corollary}

However, in practice, the covariance measure of the Hawkes process has no explicit form (except in the case of exponential kernels, see \citealp{DaFonseca2015}), so we can turn to the spectral properties of the process.
Recall that the Bartlett spectrum of a stationary multivariate point process $\mathbf N$ on $\mathbb R$ is defined as the unique, positive-definite, Hermitian measure $\mathbf \Gamma$ of auto- and cross-spectral measures $(\Gamma_{ij}(\cdot))$ such that, for any rapidly decaying function $f$ and $g$ on $\mathbb R$ \citep[chapter~8]{Daley2003}, 
\begin{equation}\label{eqn:parseval}
\mathrm{Cov}\;\bigl(N_i(f), N_j(g)\bigr) = \int_{\mathbb R} \mathcal Ff(\xi) \mathcal Fg^\ast(\xi) \Gamma_{ij}(\dd\xi),
\end{equation} 
where $\mathcal F\cdot$ denotes the Fourier transform,
\begin{equation}\label{eqn:fourier}
\mathcal Ff(\xi) = \int_{\mathbb R} e^{-2i\pi\xi u} f(u) \dd u,
\end{equation}
and $g^\ast(u) = g(-u)$, so that $\mathcal Fg^\ast$ is the complex conjugate of $\mathcal Fg$.

It turns out that the Bartlett spectrum for the multivariate Hawkes process is known, and has density given by 
\begin{equation*}
\bm\gamma(\bm\xi) = \left( \mathbf I - [\widetilde{\mathbf H}(2\pi\xi)]^\intercal \right)^{-1} \mathrm{diag}(m_1, \ldots, m_d) \left( \mathbf I - \widetilde{\mathbf H}(-2\pi\xi) \right)^{-1},
\end{equation*}
with $\mathrm{diag}(m_1, \ldots, m_d) = (\mathbf I - [\widetilde{\mathbf H}(0)]^\intercal)^{-1} \bm\eta$; discrepancies with the expression in \citet[example~8.3(c)]{Daley2003} can be attributed to the definition of the reproduction matrix $H$ and our choice of convention for the Fourier transform \eqref{eqn:fourier}.
Then the variance of $S_T$ can be calculated by the relation 
\begin{equation}\label{eqn:variance}
\sigma_T^2 = \Var S_T = \sum_{i = 1}^d \sum_{j=1}^d \int_{\mathbb R} \mathcal F f_i\big\vert_{\![0,T]}(\xi) \mathcal F f_j \big\vert_{\![0,T]} ^\ast(\xi) \gamma_{ij}(\xi) \dd\xi.
\end{equation}

\subsubsection{Two examples}
Though the integral appearing in \eqref{eqn:variance} can be difficult to compute analytically, we highlight two cases where the computation is explicit, so that it may be used in applications.

\paragraph{Asymptotically constant functions.}
Assume each $f_i$ is of the form $f_i(\cdot) = k_i + g_i(\cdot)$ with $k_i$ constants such that $\lVert \mathbf k \rVert_1 > 0$ and $g_i$ bounded and integrable functions.
Then, we find that 
\begin{equation*}
	\sigma_T^2 = T \sum_{i=1}^d \sum_{j=1}^d k_i k_j \gamma_{ij}(0) + o(T) = T \mathbf k^\intercal \bm\gamma(0) \mathbf k + o(T),
\end{equation*}
yielding that the process $\bigl\{(T \mathbf k^\intercal \bm\gamma(0) \mathbf k)^{-1/2} S_{Tu}, u \in (0,1)\bigr\}$ converges in distribution to the standard Brownian motion, and $(T \mathbf k^\intercal \bm\gamma(0) \mathbf k)^{-1/2} S_T$ to the standard normal distribution.

\begin{proof}
	Since $\mathcal F \mathbbm{1}_{[0,T]}(\xi) = T e^{-i\pi\xi T} \sinc(\pi\xi T)$, the leading term in the variance \eqref{eqn:variance} is given by
	\begin{align*}
		\sum_{i = 1}^d \sum_{j=1}^d k_i k_j \int_{\mathbb R} \mathcal F \mathbbm{1}_{[0,T]} (\xi) \mathcal F \mathbbm{1}_{[0,T]}^\ast(\xi) \gamma_{ij}(\xi) \dd\xi
		&=\sum_{i = 1}^d \sum_{j=1}^d k_i k_j T^2 \int_{\mathbb R} \sinc^2(\pi\xi T) \gamma_{ij}(\xi) \dd\xi\\
		&= \sum_{i = 1}^d \sum_{j=1}^d k_i k_j T \int_{\mathbb R} \sinc^2(\pi\omega) \gamma_{ij}(\omega/T) \dd\omega. 
	\end{align*}
	The dominated convergence theorem then yields that 
	\begin{equation*}
		\int_{\mathbb R} \sinc^2(\pi\omega)\gamma_{ij}(\omega/T) \dd\omega \xrightarrow[T\to\infty]{} \gamma_{ij}(0) \int_{\mathbb R} \sinc^2(\pi\omega) \dd\omega = \gamma_{ij}(0).
	\end{equation*}
\end{proof}

\paragraph{Asymptotically periodic functions.}
Assume each $f_i$ is of the form $f_i(\cdot) = h_i(\cdot) + g_i(\cdot)$, with $h_i$ $\tau$-periodic functions such that $\sum_{i=1}^d \int_0^\tau h_i(t) \dd t \ne 0$ and $g_i$ bounded and integrable functions. Then
\begin{equation*}
	\sigma_T^2 = \frac{T}{\tau^2} \sum_{i=1}^d \sum_{j=1}^d \sum_{k\in\mathbb Z} \mathcal Fh_i\big\vert_{\![0,\tau]}\left(\frac{k}{\tau}\right) \mathcal Fh_j\big\vert_{\![0,\tau]}^\ast \left(\frac{k}{\tau}\right) \gamma_{ij}\left(\frac{k}{\tau}\right) + o(T),
\end{equation*}
so that the process $\bigl\{ \sigma_T^{-1/2} S_{Tu}, u \in (0,1) \bigr\}$ has the desired convergence to the standard Brownian motion.

\begin{proof}
	Choose $T = n \tau$ with $n \in \mathbb N_+$.
	Then,
	\begin{align*}
		\mathcal F h_i\big\vert_{\![0,T]}(\xi) 
		&= \sum_{k=1}^n \int_{(k-1)\tau}^{k\tau} e^{-2i\pi\xi x} h_i(x) \dd x\\
		&= \sum_{k=1}^n e^{-2i\pi\xi(k-1)\tau} \mathcal F h_i\big\vert_{\![0,\tau]}(\xi),
	\end{align*}
	so that the leading term in the variance \eqref{eqn:variance} is given by
	\begin{multline*}
		\sum_{i = 1}^d \sum_{j=1}^d \int_{\mathbb R} \mathcal F h_i\big\vert_{\![0,T]} (\xi) \mathcal F h_j\big\vert_{\![0,T]}^\ast(\xi) \gamma_{ij}(\xi) \dd\xi\\
		= \sum_{i = 1}^d \sum_{j=1}^d \int_{\mathbb R} \sum_{k=1}^n \sum_{l=1}^n e^{-2i\pi\xi(k-l)\tau} \mathcal F h_i\big\vert_{\![0,\tau]} (\xi) \mathcal F h_j\big\vert_{\![0,\tau]}^\ast(\xi) \gamma_{ij}(\xi) \dd\xi.
	\end{multline*}
	A bit of calculus gives that
	\begin{align*}
		\sum_{k=1}^n \sum_{l=1}^n e^{-2i\pi\xi(k-l)\tau} 
		&= n + 2 \Re \left( e^{2i\pi\xi n\tau} \sum_{k=1}^{n-1} k e^{-2i\pi\xi k\tau} \right)\\
		&= \frac{1 - \cos(2\pi\xi n\tau)}{1 - \cos(2\pi\xi\tau)}
	\end{align*}
	so that, for any $\phi$ absolutely integrable,
	\begin{align*}
		\frac{1}{n} \int_{\mathbb R} \frac{1 - \cos(2\pi\xi n\tau)}{1 - \cos(2\pi\xi\tau)} \phi(\xi) \dd\xi 
		&= \frac{1}{n} \sum_{k \in \mathbb Z} \int_{-1/2\tau}^{1/2\tau} \frac{1 - \cos(2\pi\xi n\tau)}{1 - \cos(2\pi\xi\tau)} \phi\left(\xi + \frac{k}{\tau}\right) \dd\xi\\
		&= \frac{1}{2\pi n^2 \tau} \sum_{k \in \mathbb Z} \int_{-\pi n}^{\pi n} \frac{1 - \cos(\xi)}{1 - \cos(\xi/n)} \phi\left(\frac{\xi}{2\pi n\tau} + \frac{k}{\tau}\right) \dd\xi\\
		&= \frac{1}{2\pi\tau} \sum_{k \in \mathbb Z} \int_{-\infty}^{\infty} \frac{\1_{[-\pi n, \pi n]}(\xi)}{n^2} \frac{1 - \cos(\xi)}{1 - \cos(\xi/n)} \phi\left(\frac{\xi}{2\pi n\tau} + \frac{k}{\tau}\right) \dd\xi.
	\end{align*}
	Since the integrand is bounded above by $\xi \mapsto \lvert \phi(\xi/(2\pi n\tau) + k/\tau) \rvert$ which is integrable, and
	\begin{equation*}
		\lim_{n\to\infty} \frac{\1_{[-\pi n, \pi n]}(\xi)}{n^2} \frac{1 - \cos(\xi)}{1 - \cos(\xi/n)} \phi\left(\frac{\xi}{2\pi n\tau} + \frac{k}{\tau}\right) = 2 \, \frac{1 - \cos(\xi)}{\xi^2} \, \phi\left(\frac{k}{\tau}\right),
	\end{equation*}
	the dominated convergence theorem yields that 
	\begin{equation*}
		\lim_{n\to\infty} \int_{-\infty}^{\infty} \frac{\1_{[-\pi n, \pi n]}(\xi)}{n^2} \frac{1 - \cos(\xi)}{1 - \cos(\xi/n)} \phi\left(\frac{\xi}{2\pi n\tau} + \frac{k}{\tau}\right) \dd\xi = 2\, \phi\left(\frac{k}{\tau}\right) \int_{-\infty}^\infty \frac{1 - \cos(\xi)}{\xi^2} \dd\xi = 2\pi \phi\left(\frac{k}{\tau}\right).
	\end{equation*}
	Hence, 
	\begin{equation*}
		\xi \mapsto \frac{1}{n} \frac{1 - \cos(2\pi\xi n\tau)}{1 - \cos(2\pi\xi\tau)}~\text{converges, in the sense of distributions, to}~\xi \mapsto \frac{1}{\tau} \sum_{k\in\mathbb Z} \delta_{k/\tau}(\xi),
	\end{equation*}
	yielding the desired variance.
	
	Finally, when $T \ne n\tau$, write $T = n\tau + \varepsilon$ with $\varepsilon < \tau$ so that the leading term in the variance remains the same.
\end{proof}

\section{Proof of Theorem \ref{thm:strong_mixing}}\label{sec:proof}
Most of the proofs follow those of \citet{Cheysson2022}, extended here to the multivariate setup.
We first present the layout of the proof, introducing auxiliary results when needed, then in subsequent subsections prove these results.

\subsection{Layout of the proof}
Recall that the strong mixing coefficient of a point process $\N$ is defined by
\begin{equation*}
\alpha_{\N}(\tau) = \alpha(\mathcal{E}_{-\infty}^t, \mathcal{E}_{t+\tau}^{+\infty}) 
= \underset{t \in \mathbb{R}}{\mathrm{sup}} 
~\underset{\begin{subarray}{c}
	\mathcal{A} \in \mathcal{E}_{-\infty}^t \\
	\mathcal{B} \in \mathcal{E}_{t+\tau}^\infty
	\end{subarray}}
{\mathrm{sup}}  ~\big| \mathrm{Cov}\big(\mathbbm{1}_\mathcal{A}(\N), \mathbbm{1}_\mathcal{B}(\N)\big) \big|.
\end{equation*}
The first step of the proof is to bound the covariance of the indicator functions with that of the point process itself, using the positive association of Hawkes processes \citep[Section 2.1, key property (e)]{Gao2018}.
Recall that a $\mathcal X$-valued random variable $X$ is positively associated if for each pair of bounded, Borel measurable, non-decreasing functions $f, g: \mathcal X \to \Rr$, we have $\Cov(f(X), g(X)) \ge 0$.
This property can be used to bound the covariance of an associated point process, using Theorem 2.5 from \citet{Poinas2019}, which we extend to the multivariate setup.
\begin{proposition}
	\label{prop:poinas}
	Let $\N=(N_1, \cdots, N_d)$ be an associated multivariate point process and $A,B \subset \R$ two disjoint bounded subsets of $\R$. Let $f:\mathfrak N \to \R$ and $g:\mathfrak N \to \R$ be two functions such that $f(\N_{|A})$ and $g(\N_{|B})$ are bounded, then 
	\[\left|\Cov \Big(f(\N_{|A}), g(\N_{|B})\Big)\right| \leq \left\|f\right\|_A \left\|g\right\|_B \left|\Cov\Big(\lVert\N(A)\rVert_1, \lVert\N(B)\rVert_1\Big)\right|, \]
	where \[\left\|f\right\|_A:=\sup_{\mathbf N\in\mathfrak N}\ \sup_{\substack{1 \le i \le d\\x\in A}} |f(\mathbf N_{|A} + \bm\delta_{ix})-f(\mathbf N_A)|, \]
	and $\bm\delta_{ix}$ denotes the multivariate point process with a single atom at component $i$ and position $x$.
\end{proposition}

For the Hawkes process, this leads to the following lemma, whose proof uses the Poisson cluster structure of the process.
\begin{lemma}
	\label{lem:conditioning}
	Let $\N$ be a multivariate Hawkes process satisfying Assumption \ref{ass:spectral_radius}.
	Let $s,t,v\in\R$ and $\tau>0$ such that $s<t<t+\tau<v$, $\mathcal{A}\in\mathcal{E}_{s}^t$, $\mathcal{B}\in\mathcal{E}^v_{t+\tau}$, and $A=(s,t]$ and $B= (t+\tau,v]$. Then 
	\begin{equation}
	\label{eqn:conditioning}
	\left| \Cov\Big(\mathbbm{1}_\mathcal{A}(\N), \1_\mathcal{B}(\N)\Big) \right| \leq \sum_{i=1}^{d} \sum_{j=1}^{d} \sum_{k=1}^{d}\int \left| \Cov\Big(N_i(A \mid y), N_j(B \mid y) \Big) \right| M_k^c(dy),
	\end{equation}
	where $M_k^c$ denotes the first-order moment measure of the $k$-th component $N_k^c$ of the cluster centre process.
\end{lemma}

Without loss of generality, we consider a cluster whose immigrant is located at time $0$. Let $Z_{ki}$ denote the number of points of generation $k$ of the $i$-th component of the point process, and $Z_{ki}(A)$ those that are located in the set $A \subset \mathbb R$. 
By definition, we have 
\begin{equation*}
N_i(A \mid 0)= \sum_{k=0}^\infty Z_{ki}(A).
\end{equation*}
Then, the covariance between two components of a single branching process can be written
\begin{equation}
\label{eqn:covariance_branching_process}
\Cov\Big(N_i(A \mid 0), N_j (B \mid 0) \Big) = \sum_{k=0}^\infty\sum_{l=0}^\infty \Cov\Big(Z_{ki}(A),Z_{lj}(B)\Big).
\end{equation}

As $Z_{ki}(A),Z_{lj}(B) \geq 0$, the terms appearing in the sum on the right hand side can be bounded by
\begin{align*}
\left|\Cov(Z_{ki}(A),Z_{lj}(B))\right|  &= \left|\E\big[Z_{ki}(A) Z_{lj}(B)\big] - \E\big[Z_{ki}(A)\big]\E\big[Z_{lj}(B)\big] \right| \\
&\leq \max \Bigl(\E\bigl[Z_{ki}(A)Z_{lj}(B)\bigr],\E\bigl[Z_{ki}(A)\bigr]\E\bigl[Z_{lj}(B)\bigr]\Bigr).
\end{align*}

Let $T_{lj}^m$ denote the $m$-th time of arrival of generation $l$ and type $j$, and consider the first term appearing in the maximum above. 
Noting that for a non-negative integer $I$, $I = \sum_{n=1}^\infty \mathbbm 1_{\{I \ge n\}}$, we have that
\begin{align}
\E[Z_{ki}(A) Z_{lj}(B)] \leq \E[Z_{ki}(\mathbb R) Z_{lj}(B)] 
&=\sum_{n=1}^\infty\sum_{m=1}^\infty \E\Big[\mathbbm{1}_{\{Z_{ki}\geq n\}}\mathbbm{1}_{\{Z_{lj}\geq m\}}\mathbbm{1}_{\{T_{lj}^m\in B\}}\Big]\nonumber\\
& \leq \sum_{n=1}^\infty\sum_{m=1}^\infty \E\big[\mathbbm{1}_{\{Z_{ki}\geq n\}}\big]^{1/p} \: \E \big[\mathbbm{1}_{\{Z_{lj}\geq m\}}\big]^{1/q} \: \E\big[\mathbbm{1}_{\{T_{lj}^m\in B\}}\big]^{1/r}\nonumber\\
& = \sum_{n=1}^\infty\sum_{m=1}^\infty \P(Z_{ki}\geq n)^{1/p} \: \P(Z_{lj}\geq m)^{1/q} \: \P(T_{lj}^m\in B)^{1/r}.\label{eqn:holder}
\end{align}
where the second line follows from H\"older's inequality with any $p, q, r \ge 1$, $\frac{1}{p}+\frac{1}{q}+\frac{1}{r} = 1$. 

\begin{remark}
	Note that the use of H\"older's inequality and the need for an exponential bound on $\P(Z_{ki}\geq n)$ and $\P(Z_{lj}\geq m)$ differ from the univariate case (see \citet{Cheysson2022}), for which the times of arrival in a generation $k$ and the number of points in a generation $l$ are independent, for all $k$ and $l$.
	This is not the case in the multivariate setup, since the distribution of an arrival time may depend on its genealogy, through the types of its ancestors.
\end{remark}

We can establish a simple upper bound on $\P(T_{lj}^m\in B)$ using Markov's inequality.
\begin{lemma}
\label{lem:markov_T}
	Let $\beta > 0$ and assume that 
	\begin{equation*}
	\nu_{1+\beta} = \sup_{1 \le i,j \le d} \int_{\mathbb R} t^{1+\beta} h_{ij}^\ast(t) \mathrm dt < \infty,
	\end{equation*}
	with $h_{ij}^\ast = h_{ij} / \lVert h_{ij} \rVert_{L^1}$. Then,
	\begin{equation}
	\label{eqn:markov_T}
	\P(T_{lj}^m \in B) \le \frac{l^{1+\beta}\nu_{1+\beta}}{(t + \tau)^{1+\beta}}.
	\end{equation}
\end{lemma}

The following two lemmas form the heart of the proof, and make it distinct to the univariate case: in order to bound the right-hand side of \eqref{eqn:holder} with the fastest rate of convergence to zero, the term $\P(T_{lj}^m \in B)^{1/r} = \bigO*{\tau^{-(1+\beta)/r}}$ must converge as fast as possible, so that $r$ should be chosen as small as possible.
Consequently, in H\"older's inequality, $p$ and $q$ may be arbitrary large.
To ensure that the terms $\P(Z_{ki}\geq n)^{1/p}$ and $\P(Z_{lj}\geq m)^{1/q}$ are summable with respect to $n$ and $m$, we first prove an exponential inequality for the multitype Galton-Watson process $(Z_{ki})$, through its Laplace transform, yielding the following lemma.

\begin{lemma}
\label{lem:laplace_Z}
	Let $\mathbf u = (u_1, \ldots, u_d) \in \mathbb R^d$ with strictly positive entries. 
	Then,
	\begin{equation}
	\label{eqn:markov_Z}
	\sum_{n=1}^\infty \mathbb P(Z_{ki} \ge n)^{1/p} \le C_1 \bigl(\mathcal L\{\mathbf Z_k\}(\mathbf u) - 1\bigr)^{1/p},
	\end{equation}
	where $\mathcal L\{\mathbf Z_k\}(\mathbf u) = \E[\exp(\mathbf Z_k^\intercal \mathbf u)]$ denotes the Laplace transform of $\mathbf Z_k$ and $C_1$ denotes a constant that can be chosen uniformly with respect to $1 \le i \le d$.
\end{lemma}

\begin{remark}
	Throughout the rest of the proof, the $C_k$ are absolute constants.
\end{remark}

Plugging both \eqref{eqn:markov_T} and \eqref{eqn:markov_Z} into \eqref{eqn:holder} yields
\begin{align*}
\E[Z_{ki}(A) Z_{lj}(B)] 
&\leq \sum_{n=1}^\infty\sum_{m=1}^\infty \P(Z_{ki}\geq n)^{1/p} \: \P(Z_{lj}\geq m)^{1/q} \: \P(T_{lj}^m\in B)^{1/r}\\
&\leq C_2 \bigl(\mathcal L\{\mathbf Z_k\}(\mathbf u) - 1\bigr)^{1/p} \bigl(\mathcal L\{\mathbf Z_l\}(\mathbf u) - 1\bigr)^{1 / q} \left( \frac{l^{1+\beta}\nu_{1+\beta}}{(t+\tau)^{1+\beta}} \right)^{1 / r}.
\end{align*}
The same reasoning can be applied to $\E[Z_{ki}(A)] \E[Z_{lj}(B)]$, leading to an analogous H\"older inequality, so that we only detailed the proof for $\E[Z_{ki}(A) Z_{lj}(B)]$ which is the slightly more complicated case. Then, from \eqref{eqn:covariance_branching_process},
\begin{align}
&\left| \Cov\Big(N_i(A \mid 0), N_j (B \mid 0) \Big) \right|\nonumber\\
&\qquad\qquad \le \sum_{k=0}^\infty\sum_{l=0}^\infty C_2 \bigl(\mathcal L\{\mathbf Z_k\}(\mathbf u) - 1\bigr)^{1/p} \bigl(\mathcal L\{\mathbf Z_l\}(\mathbf u) - 1\bigr)^{1 / q} \left( \frac{l^{1+\beta}\nu_{1+\beta}}{(t+\tau)^{1+\beta}} \right)^{1 / r}.\label{eqn:covariance_branching_process_2}
\end{align}

We now focus on proving the summability of the series $\sum_k \bigl(\mathcal L\{\mathbf Z_k\}(\mathbf u) - 1\bigr)^{1/p}$.
The difficulty of this proof is algebraic in nature, as even though the spectral radius of the reproduction matrix $\mathbf M \coloneqq (\lVert h_{ij} \rVert_{L^1})_{1 \le i,j \le d}$ is strictly less than $1$ (by Assumption \ref{ass:spectral_radius}), it does not ensure that $\lVert \M \mathbf u \rVert_1 \le \lVert \mathbf u \lVert_1$ for all $\mathbf u$.
However, this holds asymptotically, such that for $n$ large, $\lVert \M^n \mathbf u \rVert_1 \le \lVert \mathbf u \rVert_1$.
Then, a recurrence on the Laplace transform of $\mathbf Z_k$ with a good choice for $\mathbf u$ yields the summability of the series.

\begin{lemma}
\label{lem:summability_Z}
	There exists $\mathbf u = (u_1, \ldots, u_d)$ with $u_j > 0$ for all $1 \le j \le d$ such that $\sum_k \bigl(\mathcal L\{\mathbf Z_k\}(\mathbf u) - 1\bigr)^{1/p}$ and $\sum_l  l^{(1+\beta) / r} \bigl(\mathcal L\{\mathbf Z_l\}(\mathbf u) - 1\bigr)^{1 / q}$ are summable.
\end{lemma}

The inequality in \eqref{eqn:covariance_branching_process_2} can then be simplified to
\begin{equation}
\label{eqn:last_bound}
\left| \Cov\Big(N_i(A \mid 0), N_j (B \mid 0) \Big) \right| \le C_3 \left( \frac{\nu_{1+\beta}}{(t+\tau)^{1+\beta}} \right)^{1 / r}.
\end{equation}

Finally, choose $r = (1+\beta) / (1+\gamma)$ and $p = q = 2\,(1+\beta)/(\beta-\gamma)$, such that $(1+\beta)/r = 1 + \gamma > 1$.
Then, plugging \eqref{eqn:last_bound} into \eqref{eqn:conditioning} yields
\begin{align*}
\left| \Cov\Big(\mathbbm{1}_\mathcal{A}(\N), \1_\mathcal{B}(\N)\Big) \right| 
&\leq \sum_{i=1}^{d} \sum_{j=1}^{d} \sum_{k=1}^{d}\int \left| \Cov\Big(N_i(A \mid y), N_j(B \mid y) \Big) \right| \eta_k \mathrm dy\\
&\leq C_4 d^3 \nu_{1+\beta}^{1/r} \int_{-\infty}^t \frac 1 {(t+\tau-y)^{1+\gamma}} \sup_{1 \le k \le d} \{\eta_k\} \  \mathrm dy\\
&\le C_5 \tau^{-\gamma}.
\end{align*}
which is the desired bound.

\subsection{Proof of Proposition \ref{prop:poinas}}

The layout of this proof is identical to that of \citet{Poinas2019}, but for clarity we detail it within the multivariate setup.
Consider a positively associated multivariate point process $\N$.
We consider the functions $f_+, f_- : \mathfrak N\to\R$, $\mathcal{E}(A)$-measurable, and $g_+, g_- : \mathfrak N\to \R$, $\mathcal{E}(B)$-measurable, defined by 
\[	\begin{cases}
f_{\pm} (\N) = f(\N_{|A}) \pm \left\|f\right\|_A \lVert\N(A)\rVert_1,\\
g_{\pm} (\N) = g(\N_{|B}) \pm \left\|g\right\|_B \lVert\N(B)\rVert_1.
\end{cases} \]
For all $1 \le i \le d$ and $x\in A$,
\begin{equation*}
f_+(\N + \bm\delta_{ix}) - f_+(\N) = f(\N_{|A} + \bm\delta_{ix}) - f(\N_{|A})   + \left\|f\right\|_A
\end{equation*}
which is positive by the definition of $\left\|f\right\|_A$. Therefore, $f_+$ is an increasing function respective to the partial order defined by $\mathbf N \le \mathbf M$ if and only if for all $1 \le i \le d$,
\begin{equation*}
\forall A \subset \mathbb R, \quad N_i(A) \le M_i(A).
\end{equation*}
With the same reasoning, $g_+$ is also increasing, and $f_-$, $g_-$ are decreasing.
While $f_+$ is not bounded, it is non-negative and almost surely finite so that by a limiting argument, using the sequence of functions $\min(f_+, k)$ when $k$ goes to infinity, the association inequality can be applied to $f_+$.
The other functions can be treated the same way, which leads to the inequalities
\begin{equation*}
\Cov \big(f_{+}(\N),g_{+}(\N)\big)\geq 0 \hspace{0.3cm}\text{and} \hspace{0.3cm}\Cov \big(f_{-}(\N),g_{-}(\N)\big)\geq 0.
\end{equation*}
We can then expand these expressions
\begin{multline*}
\Cov \big(f_+(\N),g_+(\N)\big) = \Cov \big(f(\N_{|A}), g(\N_{|B})\big) + \left\|f\right\|_A\left\|g\right\|_B  \Cov\big(\lVert\N(A)\rVert_1,\lVert\N(B)\rVert_1\big)\\
+ \left\|g\right\|_B \Cov \big(f(\N_{|A}), \lVert\N(B)\rVert_1\big) + \left\|f\right\|_A \Cov \big(\lVert\N(A)\rVert_1,g(\N_{|B})\big) 
\end{multline*}
and 
\begin{multline*}
\Cov \big(f_-(\N),g_-(\N)\big) = \Cov \big(f(\N_{|A}), g(\N_{|B})\big) + \left\|f\right\|_A\left\|g\right\|_B  \Cov\big(\lVert\N(A)\rVert_1),\lVert\N(B)\rVert_1\big)\\ 
-  \left\|g\right\|_B \Cov \big(f(\N_{|A}), \lVert\N(B)\rVert_1\big) - \left\|f\right\|_A \Cov \big(\lVert\N(A)\rVert_1,g(\N_{|B})\big).
\end{multline*}
Adding these two expressions together yields the lower bound:
\begin{equation}\label{equa:1}
\Cov \big(f(\N_{|A}), g(\N_{|B})\big) \geq - \left\|f\right\|_A\left\|g\right\|_B \Cov\big(\lVert\N(A)\rVert_1,\lVert\N(B)\rVert_1\big).
\end{equation} 
The upper bound is obtained similarly by replacing $f$ by $-f$ in the previous expression.

\subsection{Proof of Lemma \ref{lem:conditioning}}

Using Proposition \ref{prop:poinas} with $f = \mathbbm{1}_\mathcal{A}$ and $g = \mathbbm{1}_\mathcal{B}$, we have 
\begin{align*}
\left| \Cov\Big(\mathbbm{1}_\mathcal{A}(\N), \mathbbm{1}_\mathcal{B}(\N)\Big)\right| &\leq  \left|\Cov\Big(\lVert\N(A)\rVert_1, \lVert\N(B)\rVert_1\Big)\right|\\
&\le \sum_{i=1}^d \sum_{j=1}^d \left|\Cov\Big(N_i(A), N_j(B)\Big)\right|.
\end{align*}
Then, denoting by $M_k^c$ the first-order moment measure of the $k$-th component $N_k^c$ of the cluster centre process, and by $C_{kl}^c$ the covariance measure between the $k$-th and $l$-th components of the cluster centre process ($1 \le k, l \le d$), we have, by conditioning on the cluster centre process (see for example \citealp[Exercise 6.3.4]{Daley2003}),
\begin{multline*}
\Cov\Big(N_i(A), N_j(B)\Big)= \sum_{k=1}^{d}\int \Cov \Big(N_i(A \mid y),N_j(B \mid y)\Big)M_k^c(dy)\\
+ \sum_{k=1}^{d}\sum_{l=1}^{d}\int \int \E(N_i(A) \mid x) \E(N_j(B) \mid y)C_{kl}^c(\mathrm dx\times \mathrm dy).
\end{multline*}
Since the centre process is Poisson, $C_{kl}^c \equiv 0$ 
and the second term vanishes.      

\subsection{Proof of Lemma \ref{lem:markov_T}: an upper bound for $\P(T_{lj}^m \in B)$}
Consider the inter-arrival times between a parent and its offsprings.
Since an individual $T_{ki}^n$ of generation $k$ and type $i$ generates offsprings of generation $k+1$ and type $j$ according to an inhomogeneous Poisson process with intensity $h_{ij}(\cdot - T_{ki}^n)$, these inter-arrival times are independent of each other, and their distributions only depend on the type $i$ of the parent and the type $j$ of the offspring.
Consequently, any single arrival time $T_{lj}^m$ can be written as the sum of $l$ independent inter-arrival times,
\begin{equation*}
T_{lj}^m = \sum_{\iota = 1}^l \Delta_\iota(T_{lj}^m),
\end{equation*}
where the density functions of the $\Delta_\iota(T_{lj}^m)$ are amongst the $(h_{ij}^\ast)_{i,j}$.

Hölder's inequality then gives,
\begin{equation*}
T_{lj}^m = \sum_{\iota = 1}^l \Delta_\iota(T_{lj}^m) \le \left( \sum_{\iota = 1}^l 1 \right)^{\beta / (1 + \beta)} \times \left( \sum_{\iota = 1}^l \left( \Delta_\iota(T_{lj}^m) \right)^{1+\beta} \right)^{1 / (1+\beta)},
\end{equation*}
and by independence of the $\Delta_\iota(T_{lj}^m)$,
\begin{equation*}
\mathbb E \left[ \sum_{\iota = 1}^l \left( \Delta_\iota(T_{lj}^m) \right)^{1+\beta} \right] \le l \, \sup_{1 \le \iota \le l} \mathbb E \left[ \left( \Delta_\iota(T_{lj}^m) \right)^{1+\beta} \right] \le l \, \nu_{1+\beta}.
\end{equation*}
Hence,
\begin{equation*}
\mathbb E\left[ \left( T_{lj}^m \right)^{1+\beta}\right] \le l^{1+\beta} \nu_{1+\beta}.
\end{equation*}

Finally, we can apply Markov's inequality to $\P(T_{lj}^m \in B)$, recalling that $B = (t+\tau, v]$:
\begin{equation*}
\P(T_{lj}^m \in B) \le \P(T_{lj}^m > t + \tau) \le \frac{l^{1+\beta}\nu_{1+\beta}}{(t + \tau)^{1+\beta}}.
\end{equation*}

\subsection{Proof of Lemma \ref{lem:laplace_Z}: exponential inequality for $\P(Z_{ki}\geq n)$ and $\P(Z_{lj}\geq m)$}\label{sec:proof_exp_ineq}
First recall that the Laplace transform of a Poisson variable $X$ with intensity $\lambda$ is given by
\begin{equation}
\label{eqn:laplace}
\mathcal L\{X\}(u) = \mathbb E[e^{uX}] = e^{\lambda(e^u - 1)}.
\end{equation}
For a $\mathbb N$-valued random variable $Y$ and a constant $\alpha \ge 0$, define the reproduction operator $\circ$ by
\begin{equation*}
\alpha \circ Y = \sum_{n=1}^Y \xi_n^{(\alpha)},
\end{equation*}
where $\xi_n^{(\alpha)} \overset{\text{i.i.d.}}{\sim} \mathcal P(\alpha)$, independently of $Y$.

Denoting by $\mathbf Z_k = (Z_{k1}, \ldots, Z_{kd})^\intercal$ the random vector for the number of points of generation $k$ for each component of the branching process, we have that conditionally on $\mathbf Z_{k-1}$, the variables $Z_{kj}$ are Poisson with intensity $(\mathbf Z_{k-1}^\intercal \mathbf M)_j = \sum_{i=1}^d Z_{k-1,i} M_{ij}$ ($k > 0$, $1 \le j \le d$).
More specifically, we have that
\begin{equation*}
Z_{kj} = \sum_{i=1}^d M_{ij} \circ Z_{k-1,i} = \sum_{i=1}^d \sum_{n = 1}^{Z_{k-1,i}} \xi_n^{(M_{ij})}
\end{equation*}
where $\xi_n^{(M_{ij})} \overset{\text{i.i.d.}}{\sim} \mathcal P(M_{ij})$ independently of all $Z_{k-1,i}$.

For all $\mathbf{u}= (u_1, \cdots,u_d) \in \R^d$, consider now the Laplace transform of $\Z_k$,
\begin{align}
\mathcal L\{\mathbf Z_k\}(\mathbf u) = \E\left[\exp(\Z_k^\intercal \mathbf{u})\right] 
&= \E\left[\exp\left(\sum_{j=1}^{d} u_j Z_{kj}\right)\right]\nonumber \\ 
&= \E\left[\exp\left(\sum_{j=1}^{d} \sum_{i=1}^{d} u_j M_{ij} \circ Z_{k-1,i}\right)\right]\nonumber \\ 
&= \E\left[\exp\left(\sum_{j=1}^{d} \sum_{i=1}^{d} \sum_{n=1}^{Z_{k-1,i}} u_j \xi_n^{(M_{ij})}\right)\right] \nonumber\\
&= \E\left[ \prod_{j=1}^{d} \prod_{i=1}^{d} \prod_{n=1}^{Z_{k-1,i}} \E \left[ \exp \left(u_j \xi_n^{(M_{ij})}\right) \,\middle\vert\, \mathbf Z_{k-1} \right] \right] \nonumber\\
&= \E\left[ \prod_{j=1}^{d} \prod_{i=1}^{d} \prod_{n=1}^{Z_{k-1,i}} \exp \left( M_{ij} (e^{u_j} - 1) \right) \right] \label{eqn:laplace_used}\\
&= \E\left[ \exp \left( \sum_{j=1}^{d} \sum_{i=1}^{d} Z_{k-1,i} M_{ij} (e^{u_j} - 1) \right) \right] \nonumber\\
&= \E\left[ \exp \left( \mathbf Z_{k-1}^\intercal \mathbf M (\mathbf e^\mathbf u - \mathbf 1) \right) \right], \nonumber
\end{align}
where \eqref{eqn:laplace_used} was obtained using \eqref{eqn:laplace} and the fact that $\xi_n^{(M_{ij})}$ and $\mathbf Z_{k-1}$ are independent, and $\mathbf e^{\mathbf u}$ denotes the component-wise exponentiation of $\mathbf u$.

Define the function $\mathbf g$ by $\mathbf g(\mathbf u) = \mathbf M(\mathbf e^{\mathbf u} - \mathbf 1)$.
By recurrence, we therefore have
\begin{equation}
\label{eqn:laplace_recurrence}
\mathcal L\{\mathbf Z_k\}(\mathbf u) = \mathcal L\{\mathbf Z_{k-1}\}(\mathbf g(\mathbf u)) = \ldots = \mathcal L\{\mathbf Z_0\}(\mathbf g^k(\mathbf u)),
\end{equation}
where $\mathbf g^k$ stands for the $k$-fold composition of $\mathbf g$ with itself.

Finally, we can apply Markov's inequality to $Z_{ki}$ for any $1 \le i \le d$,
\begin{equation*}
\mathbb P(Z_{ki} \ge n) = \mathbb P(e^{u_i Z_{ki}} - 1 \ge e^{u_i n} - 1) \le \frac{\mathbb E e^{u_i Z_{ki}} - 1}{e^{u_i n} - 1} \le \frac{\mathcal L\{\mathbf Z_k\}(\mathbf u) - 1}{e^{u_i n} - 1},
\end{equation*}
where $\mathbf u = (u_1, \ldots, u_d)$ was chosen with strictly positive entries. 
Then, as $\sum_n (e^{u_i n} - 1)^{-1/p}$ is summable, $\sum_n \mathbb P(Z_{ki} \ge n)^{1/p}$ is also summable and
\begin{equation*}
\sum_{n=1}^\infty \mathbb P(Z_{ki} \ge n)^{1/p} \le C_1 \left(\mathcal L\{\mathbf Z_k\}(\mathbf u) - 1\right)^{1/p},
\end{equation*}
where $C_1$ is an absolute constant, which can be chosen uniformly with respect to $1 \le i \le d$.

\subsection{Proof of Lemma \ref{lem:summability_Z}: summability of $\sum_k (\mathcal L\{\mathbf Z_k\}(\mathbf u) - 1)^{1/p}$}\label{sec:proof_summability}
Define $\delta$ such that $1 < \delta < \rho^{-1}$, where $\rho = \mathrm{Sp}(\mathbf M) < 1$ denotes the spectral radius of the reproduction matrix $\mathbf M$.
Then there is a $u_0 > 0$ such that, for all $u \in [0, u_0]$, $e^u - 1 \le \delta u$.
Applying this to the vector $\mathbf u = (u_1, \ldots, u_d)$ with $u_j \in [0, u_0]$ for all $1 \le j \le d$, we find that
\begin{equation}
\label{eqn:recurrence_relation}
\mathbf g(\mathbf u) = \mathbf M(\mathbf e^{\mathbf u} - \mathbf 1) \le \delta \mathbf M \mathbf u,
\end{equation}
where the inequality holds term by term.

\begin{lemma}
	\label{lem:g_bound}
	There exists $\mathbf u = (u_1, \ldots, u_d)$ with $u_j > 0$ for all $1 \le j \le d$ such that, for all $k \in \mathbb N$, 
	\begin{equation*}
	\mathbf g^k(\mathbf u) \le \delta^k \mathbf M^k \mathbf u.
	\end{equation*}
	Furthermore, there exists a constant $0 < c < 1$ and an integer $k_0 > 0$ such that, for all $k \ge k_0$, 
	\begin{equation}
	\label{eqn:g_convergence}
	\lVert \mathbf g^k(\mathbf u) \rVert_1 \le c^k \, \lVert \mathbf u \rVert_1.
	\end{equation}
\end{lemma}
\begin{proof}
	First note that, since $\rho < 1$, then for any matrix norm $\Vert \cdot \rVert$,
	\begin{equation*}
	\lim_{n \to \infty} \lVert \mathbf M^n \rVert^{1/n} = \rho.
	\end{equation*}
	Define $\varepsilon = 1/2\, (\delta^{-1} - \rho) > 0$ such that $c_\varepsilon = \delta (\rho + \varepsilon) < 1$.
	Then there exists $k_0 > 0$ such that, for all $k \ge k_0$,
	\begin{equation*}
	\lVert \mathbf M^k \rVert_1 \le (\rho + \varepsilon)^k.
	\end{equation*}
	Therefore, for all $k \ge k_0$,
	\begin{equation*}
	\lVert \delta^k \mathbf M^k \mathbf u \rVert_1 \le \delta^k (\rho + \varepsilon)^k \lVert \mathbf u \rVert_1 = c_\varepsilon^k \, \lVert \mathbf u \rVert_1 \xrightarrow[k \to \infty]{} 0.
	\end{equation*}
	Since $\lVert \delta^k \mathbf M^k \mathbf u \rVert_\infty \le \lVert \delta^k \mathbf M^k \mathbf u \rVert_1$, the sequence $(\lVert \delta^k \mathbf M^k \mathbf u \rVert_\infty)_{k \in \mathbb N}$ is bounded for any $\mathbf u$.
	Choose $\mathbf u$ with strictly positive entries such that $\lVert \delta^k \mathbf M^k \mathbf u \rVert_\infty \le u_0$ for all $k\in \mathbb N$ (such a $\mathbf u$ exists since the function $\mathbf u \mapsto \sup_{k \in \mathbb N} \lVert \delta^k \mathbf M^k \mathbf u \rVert_\infty$ is continuous and non-decreasing).
	Then, a straightforward recurrence from \eqref{eqn:recurrence_relation} yields the result.
\end{proof}

Plugging \eqref{eqn:g_convergence} into the recurrence relation for the Laplace transform of $\mathbf Z_k$, \eqref{eqn:laplace_recurrence}, for a $\mathbf u$, $c$ and $k_0$ defined as in Lemma \ref{lem:g_bound} and $k \ge k_0$, we find that
\begin{equation}
\label{eqn:laplace_to_gk}
\mathcal L\{\mathbf Z_k\}(\mathbf u) = \mathbb E[\exp(\mathbf Z_0^\intercal \mathbf g^k(\mathbf u))] \le \exp \big(\, \lVert \mathbf g^k(\mathbf u) \rVert_1\big) \le \exp \big(c^k \lVert \mathbf u \rVert_1\big),
\end{equation}
noting for the first inequality that for the branching process, $\mathbf Z_0$ is a vector with only one non-zero component, equal to one.
Therefore, when $k \ge k_0$, we have that
\begin{equation*}
\lvert\mathcal L\{\mathbf Z_k\}(\mathbf u) - 1\rvert \le \lvert \exp \big(c^k \lVert \mathbf u \rVert_1\big) - 1 \rvert \underset{k \to \infty}{=} c^k \lVert \mathbf u \rVert_1 + o(c^k).
\end{equation*}
Hence, $\sum_k \big(\mathcal L\{\mathbf Z_k\}(\mathbf u) - 1\big)^{1 / p}$ is summable.
Similarly, the series $\sum_l  l^{(1+\beta) / r} \big(\mathcal L\{\mathbf Z_l\}(\mathbf u) - 1\big)^{1 / q}$ is summable.

\section{Proof of Proposition \ref{prop:fclt}}\label{sec:proof_fclt}


Let us first write the functional central limit theorem proven in \citet{Merlevede2020} in the context of \citeauthor{Rosenblatt1956}'s traditional $\alpha$-mixing coefficient \citep{Rosenblatt1956}, for non-triangular sequences.
Consider a sequence $\{X_k\}_{k \ge 1}$ of centered ($\mathbb E[X_k] = 0$) variables, with partial sums $S_n = \sum_{k=1}^n X_k$.
Let $\sigma_n^2 = \Var S_n$.
For $0 \le t \le 1$, define 
\begin{equation*}
v_n(t) = \inf \left\{ m: 1 \le m \le n, \sigma_m^2 / \sigma_n^2 \ge t\right\} \quad \text{and} \quad W_n(t) = \sum_{k=1}^{v_n(t)} X_k.
\end{equation*}
Consider, for $\delta > 0$, the conditions
\begin{equation}\label{eqn:merlevede_cond2}
\sup_{n\ge 1}\, \frac{1}{\sigma_n^2} \sum_{k=1}^n E[X_k^2] < \infty, \quad \text{and} \quad \frac{1}{\sigma_n^{2+\delta}} \max_{1 \le k \le n} \mathbb E[ \lvert X_k \rvert^{2+\delta} ] \to 0,
\end{equation}
and 
\begin{equation}\label{eqn:merlevede_cond3}
\sup_{n \ge 1}\, \frac{1}{\sigma_n^2} \sum_{k = 1}^n \lVert X_k \rVert_{2+\delta}^2 < \infty \quad \text{and} \quad \sum_{\tau \ge 1} \tau^{2/\delta} \alpha(\tau) < \infty.
\end{equation}
\begin{theorem}[\citealp{Merlevede2020}]\label{thm:merlevede_fclt}
	Assume that conditions \eqref{eqn:merlevede_cond2} and \eqref{eqn:merlevede_cond3} are satisfied for some $\delta > 0$.
	Then, $\{\sigma_n^{-1} W_n(u), u \in (0,1]\}$ converges in distribution in $D([0,1])$ (equipped with the uniform topology) to the standard Brownian motion.
\end{theorem}

We look to apply this result to functions of Hawkes processes.
Define the sequence $\{X_k\}_{k \ge 1}$ as
\begin{equation*}
	X_k = \sum_{i=1}^d \int_{k-1}^k f_i(u) \bigl( N_i(\dd u) - m_i \dd u \bigr),
\end{equation*}
so that its partial sums $S_n$ coincide with the statistic defined in \eqref{eqn:partial_sums} when $T = n$ is integer.


First note that
\begin{align*}
	\sum_{k=1}^n \mathbb E[X_k^2] 
	&= \sum_{k=1}^n \sum_{i = 1}^d \sum_{j=1}^d \iint_{[k-1, k]^2} f_i(u)f_j(v) C_{ij}(\dd u \times \dd v)\\
	&\le \sum_{k=1}^n \sum_{i = 1}^d \sum_{j=1}^d \lVert f_i \rVert_\infty \lVert f_j \rVert_\infty \iint_{[0, 1]^2} C_{ij}(\dd u \times \dd v) \eqcolon n \kappa_1,
\end{align*}
where the inequality comes from the boundedness of $f_i$ and the stationarity of $\N$.
Then, 
\begin{equation*}
\sup_{n \ge 1}\, \frac{1}{\sigma_n^2} \sum_{k = 1}^n \mathbb E[X_k^2] \le \sup_{n \ge 1}\, \frac{n \kappa_1}{n\sigma^2 + o(n)} < \infty.
\end{equation*}
and the left part of condition \eqref{eqn:merlevede_cond2} is fulfilled.
Next, we have that
\begin{align*}
	\mathbb E[\lvert X_k \rvert^{2+\delta}] 
	& \le d^{1+\delta} \sum_{i=1}^d \mathbb E \left\lvert \int_{n-1}^n f_i(u) \bigl(N_i(\dd u) - m_i \dd u\bigr) \right\rvert^{2+\delta}\\
	& \le (2d)^{1+\delta} \sum_{i=1}^d \left( \mathbb E \left\lvert \int_{n-1}^n f_i(u) N_i(\dd u) \right\rvert^{2+\delta} + \mathbb E \left\lvert \int_{n-1}^n f_i(u) m_i \dd u \right\rvert^{2+\delta} \right)\\
	& \le (2d)^{1+\delta} \sum_{i=1}^d \lVert f_i \rVert_\infty^{2+\delta} \left( \mathbb E \left[ \bigl\lvert N_i((0,1]) \bigr\rvert^{2+\delta} \right] + m_i^{2+\delta} \right) \eqcolon \kappa_2,
\end{align*}
where the first two inequalities follow from Jensen's inequality and the third from the boundedness of $f_i$ and the stationarity of $N_i$.
Then, as $N_i$ admits finite exponential moments as a Poisson cluster of subcritical multitype Poisson branching processes, we find that $\kappa_2 < \infty$, and since $\sigma_n^2 = n \sigma^2 + o(n)$ as $n \to \infty$, the right part of condition \eqref{eqn:merlevede_cond2} is verified.


For the left hand side condition in \eqref{eqn:merlevede_cond3}, we write
\begin{equation*}
	\sup_{n \ge 1}\, \frac{1}{\sigma_n^2} \sum_{k = 1}^n \lVert X_k \rVert_{2+\delta}^2
	\le \sup_{n \ge 1}\,\frac{n \kappa_2^{2/(2+\delta)}}{n \sigma^2 + o(n)} < \infty,
\end{equation*}
and the right hand side is satisfied when $(\beta - 1)\delta > 2$ according to Theorem \ref{thm:strong_mixing}.

Finally, Proposition \ref{prop:fclt} follows by applying the Theorem \ref{thm:merlevede_fclt} to $\{X_k\}_{k \ge 1}$, and whenever $T \not\in \mathbb N$ using Slutsky's lemma for $S_T = S_{\lfloor T \rfloor} + R(T)$ with $R(T) = o_{\mathbb P}(\sigma_T)$.

\bibliography{references}
\bibliographystyle{abbrvnatnourl}

\end{document}